\documentclass[12pt]{amsart}
\usepackage{amsfonts}
\usepackage{amsfonts,latexsym,rawfonts,amsmath,amssymb,amsthm}
\usepackage{amsmath,amssymb,amsthm,amscd,esint}
\usepackage[plainpages=false]{hyperref}
\usepackage{appendix}

\usepackage{tikz}
\usepackage{mathdots}
\usepackage{yhmath}
\usepackage{cancel}
\usepackage{siunitx}
\usepackage{multirow}
\usepackage{tabularx}
\usepackage{extarrows}
\usepackage{booktabs}
\usetikzlibrary{fadings}
\usetikzlibrary{patterns}
\usetikzlibrary{shadows.blur}
\usetikzlibrary{shapes}

\usepackage{graphicx}
\usepackage{float}
\usepackage{subfigure}

\RequirePackage{color}

\numberwithin{equation}{section}

\newcommand{\beq}{\begin{equation}}
\newcommand{\eeq}{\end{equation}}
\newcommand{\beqs}{\begin{eqnarray*}}
\newcommand{\eeqs}{\end{eqnarray*}}
\newcommand{\beqn}{\begin{eqnarray}}
\newcommand{\eeqn}{\end{eqnarray}}
\newcommand{\beqa}{\begin{array}}
\newcommand{\eeqa}{\end{array}}

\newcommand{\p}{\partial}

\newcommand{\Om}{\Omega}
\newcommand{\pom}{{\p\Om}}

\newcommand\tbbint{{-\mkern -16mu\int}}  \newcommand\dbbint{{-\mkern -19mu\int}}   \newcommand\bbint{ {\mathchoice{\dbbint}{\tbbint}{\tbbint}{\tbbint}} }

\newtheorem{prop}{Proposition}[section]
\newtheorem{theo}[prop]{Theorem}
\newtheorem{lem}[prop]{Lemma}

\newtheorem{cor}[prop]{Corollary}
\newtheorem{rem}[prop]{Remark}

\allowdisplaybreaks
\begin{document}
\baselineskip=16.4pt
\parskip=3pt

\title{H\"older regularity of Dirichlet problem for 
the complex Monge-Amp\`ere equation}
\author{Yuxuan Hu and Bin Zhou}

\begin{abstract} 
We study the Dirichlet problem for the complex Monge-Amp\`ere equation on a strictly pseudo-convex domain in $\mathbb C^n$ or a Hermitian manifold. Under the condition that the right-hand side lies in $L^p$ function and the boundary data are H\"older continuous, we prove the global H\"older continuity of the solution.
\end{abstract}

\address{Yuxuan Hu:School of Mathematical Sciences, Peking University, Beijing 100871, China.} 
\email{huyx@pku.edu.cn}

\address{Bin Zhou:School of Mathematical Sciences, Peking University, Beijing 100871, China.} 
\email{bzhou@pku.edu.cn}

\thanks {Partially supported by National Key R$\&$D Program of China 2023YFA009900 and NSFC Grant 12271008.}

\subjclass[2000]{Primary: 32W20; Secondary: 35J60.}


\maketitle

\section{Introduction}

Let  $\Omega\subset \mathbb C^n$ be a bounded pseudoconvex domain. Assume $\varphi\in C(\partial\Omega)$ and
$f\in L^p(\Omega)$.
We consider the Dirichlet problem
\begin{equation}
\label{LCMA-dir}
  \left\{ 
  \begin{alignedat}{2}(dd^cu)^n~& =f\,d\mu~&&\text{\ in} ~\ \ \Omega, \\\
 u ~&=\varphi~&&\text{\ on}~\ \  \pom,
\end{alignedat}
\right.
\end{equation}
where $u\in C(\bar\Omega)\cap PSH(\Omega)$ and $d\mu$ denotes the Lebesgue measure.
When $f\in C(\bar\Omega)$,  the existence of continuous weak solutions to \eqref{LCMA-dir} was established in \cite{Br, BT1, Wa}.
If, in addition, $f^{\frac{1}{n}}\in C^{\alpha}(\bar\Omega)$ for some $0<\alpha\leq 1$ and $\varphi\in C^{2\alpha}(\partial\Omega)$, it is shown in \cite{BT1} that the solution $u$ belongs to $C^{\alpha}(\bar\Omega)$. In a seminal work \cite{K1, K2}, Ko\l odziej proved that the Dirichlet problem still admits a continuous solution when $f\in L^p(\Omega)$ for $p>1$. Later,  Guedj-Ko\l odziej-Zeriahi \cite{GKZ} showed that $u\in C^\alpha(\bar\Omega)$ for $\alpha<\frac{2}{np^*+1}$ under the assumptions that $\varphi\in C^{1,1}(\bar\Omega)$, $f\in L^p(\Omega)$ and $f$ is  bounded near $\partial\Omega$. The requirement that $f$ be bounded near $\partial\Omega$ was subsequently removed by
\cite{Ch1}. In this context, the counterexample \cite{Pl} demonstrates the H\"older exponent can not exceed $\frac{2}{np^*}$.

The recent examples in \cite{WW} show that without regularity assumptions on the boundary value, the solution may fail to be Dini continuous even when the right-hand side $f \equiv 1$. This naturally raises the question of whether the solution remains Hölder continuous when $\varphi$ is only H\"older continuous.
We begin by recalling the approach in \cite{GKZ}.  For any $\epsilon>0$, define 
\[\Omega_\epsilon:=\{x\in \Omega|\, \text{dist}(x,\pom)>\epsilon\}\]
and
\begin{equation}
\hat{u}_{\epsilon}(x):=\bbint_{|\zeta-x|\leq \epsilon}u(\zeta)\,d\mu,\ x\in\Omega_{\epsilon}.\label{mo2}
\end{equation}
If $u$ is plurisubharmonic in $\Omega$,  then so is $\hat{u}_{\epsilon}$.  For simplicity, we denote 
\[\gamma_0=\gamma_0(p)=\frac{1}{p^*+1},\ \ \gamma_n=\gamma_n(p)=\frac{1}{np^*+1}\]
for $p>1$. 
To establish global H\"older estimates, it suffices to prove that $u$ is Hölder continuous near the boundary and to bound the $L^{\infty}$ norm of $\hat{u}_{\epsilon}-u$ by a constant multiple of ${\epsilon}^{\alpha}$
(see Lemma \ref{elementarylemma}). 
The key elements in \cite{GKZ} include:
\begin{enumerate}
\item [(1)] Construction of a H\"older continuous barrier, which implies boundary Hölder estimates for the solution;
\item [(2)] Reduction of the estimate for  $\sup_{\Omega_\epsilon}\{\hat{u}_{\epsilon}-u\}$ via stability estimates for the complex Monge-Ampère equation to an estimate of $\|\hat{u}_{\epsilon}-u\|_{L^1(\Omega_\epsilon)}^\gamma$, with $0\leq \gamma<\gamma_n$;
\item [(3)] An estimate of $\|\hat{u}_{\epsilon}-u\|_{L^1(\Omega_\epsilon)}$ in terms of the total mass of $\triangle u$, i.e.,
\beq\label{L1-lap}
\|\hat{u}_{\epsilon}-u\|_{L^1(\Omega_\epsilon)}\leq C\epsilon^2 \|\triangle u\|_{L^1(\Omega)}.
\eeq
\end{enumerate}

In  \cite{BKPZ}, by introducing a technique to truncate the mass of $\triangle u$, the authors established
\beq\label{L1-cut}
\|\hat{u}_{\frac{\epsilon}{2}}-u\|_{L^1(\Omega_\epsilon)}\leq C\epsilon^{1-\delta}\|(-\rho)^{1+\delta} \triangle u\|_{L^1(\Omega_{\frac{\epsilon}{2}})}
\eeq
for $0 < \delta < 1$, where $\rho$ is the defining function of $\Omega$. By estimating the right-hand side, they proved that the solution is $C^{\min\{\frac{\alpha}{m}, \frac{2\gamma}{m}\}}$-H\"older  continuous for $\gamma<\gamma_n$, when $f\in L^p(\Omega)$ and $\varphi\in C^\alpha(\bar\Omega)$ on a smooth pseudoconvex domain of finite type $m$ with $m\geq 2$. Subsequently, adapting the argument in \cite{BKPZ}, Charabati \cite{Ch2} obtained improved Hölder exponents:  $u\in C^{\min\{\frac{\alpha}{2}, \gamma\}}(\bar{\Omega})$ for any $0<\gamma<\gamma_n$ on smooth strongly pseudoconvex domains, and $u\in C^{\min\{\frac{\alpha}{4}, \frac{\gamma}{2}\}}(\bar{\Omega})$ on strongly pseudoconvex Lipschitz domains. 

In this paper, we employ alternative methods to improve the Hölder exponent. For the boundary Hölder regularity, we introduce a new construction of the barrier function, leading to a different exponent compared to \cite{Ch2}. In estimating $\|\hat{u}_{\epsilon}-u\|_{L^1(\Omega_\epsilon)}$,  we adopt a more elementary approach that avoids relying on the mass of $\triangle u$ and makes greater use of the boundary H\"older estimates. Specifically, we prove
\[\left|\left|{\hat{u}}_{\epsilon}-u\right|\right|_{L^1(\Omega_{\epsilon})}\leq C{\epsilon}^{1+\beta}.\]
Moreover, our results extend to complete Hermitian manifolds. In this setting, we use the regularizations from \cite{D, BD} in place of $\hat{u}_\epsilon$.
Our main result is as follows:

\begin{theo}\label{hold1}
Let $(X,\omega)$ be a complete Hermitian manifold, and let $\Omega$ be a relatively compact smooth strictly pseudo-convex open subset of $X$.
Suppose $0\leq f\in L^{p}(\Omega,\omega^n)$ for $p>1$ and $\varphi\in C^{\alpha}(\pom)$.
Let $u$ be a solution to the Dirichlet problem:
\begin{equation}
\label{CMA-manifold}
  \left\{ 
  \begin{alignedat}{2}(dd^cu)^n~& =f\omega^n~&&\text{\ in} ~\ \ \Omega, \\\
 u ~&=\varphi~&&\text{\ on}~\ \  \pom.
\end{alignedat}
\right.
\end{equation}
Then  for any $0<\gamma,\gamma'<\gamma_n$, $0<\gamma''<\gamma_0$, we have $u\in C^{\alpha'}(\bar{\Omega})$ with 
\[\alpha'=\min\{\beta,(1+\beta)\gamma\},\] where
\beq\label{beta}
\beta=\max\{\min\{\gamma'',\frac{\alpha}{2+\alpha}\},\ \min\{\frac{\alpha}{2},\gamma'\}\}.
\eeq
Furthermore, there exists a constant $C>0$, 
which depends only on $n$, $p$, $\alpha$, $\beta$, $\gamma$, $\left\|\varphi\right\|_{C^{\alpha}(\pom)}$ and $\|f\|_{L^p(\Om)}$ such that
\[\|u\|_{C^{0,\alpha'}(\Omega)}\leq C.\]
\end{theo}

\vskip 20pt

\section{H\"older continuity via regularization}\label{reg-lem}

Regularization techniques are extensively used in the study of regularity for the complex Monge-Amp\`ere equation; see, for example, \cite{BD, GKZ, DDG+, KN1, KN2}. A detailed characterization of the modulus of continuity for subharmonic functions can be found in \cite{Z}. In this section, we present an elementary lemma on the characterization of H\"older continuity. Notably, the assumption of subharmonicity is removed, which may make the lemma applicable in broader settings.

\subsection{The flat case}
Let $\Omega$ be a bounded domain in $\mathbb{R}^n$ and let $u \in C(\overline{\Omega})$. Consider a function $\eta \in L^1(\mathbb{R}^n)$ that is a non-negative Borel function with compact support in the ball $B_R(0)$, normalized by $\displaystyle\int_{\mathbb{R}^n}\eta \,d\mu=1$.
 Assume there exist an open set $U \subset B_R(0)$ and a constant $\delta > 0$ such that $\eta \geq \delta$ on $U$. Define the regularization  $u_\epsilon = u*\eta_{\epsilon}$ on $\Omega_\epsilon$, where $\eta_{\epsilon}(x)=\frac{1}{\epsilon^n}\eta(\frac{x}{\epsilon})$. 

\begin{lem}\label{elementarylemma}
Assume there exist constants $\epsilon_0 > 0$, $C_1, C_2 > 0$, $R > 1$, and $\alpha \in (0,1)$ such that the following hold:
\begin{enumerate}
\item [(1)] $\left|u(x)-u(y)\right|\leq C_1\left|x-y\right|^{\alpha},\ \ \forall x\in \Omega,\ y\in\pom$;
\item [(2)] $\forall \epsilon\in(0,{\epsilon}_0)$, we have
$$\left|u_{\epsilon}(x)-u(x)\right|\leq C_2{\epsilon}^{\alpha},\ \,\forall x\in {\Omega}_{R\epsilon}.$$
\end{enumerate}
Then there exists $C>0$, depending only on $\epsilon_0$, $C_1$, $C_2$, R, $diam(\Omega)$, $\alpha$, and $\eta$, such that  
\[\left|u(x)-u(y)\right|\leq C\left|x-y\right|^{\alpha},\ \  \forall x,y\in\Omega.\]
\end{lem}
\begin{proof}
Without loss of generality, we may assume that $U$ contains a ball of radius $3$; if not, we can replace $\eta$ by  a suitable dilation $\eta_{\delta_0}$ for some $\delta_0>0$.
For $0<r\leq diam(\Omega)$, define the modulus of continuity
\[\omega(r)=\sup_{x,y\in\Omega,\left|x-y\right|\leq r}\left|u(x)-u(y)\right|.\]
Now, fix $r\leq {\epsilon}_0$ and consider $x,y\in\Omega$ such that $\left|x-y\right|\leq r$.

First, suppose $\mathrm{dist}(x, \partial\Omega) \leq Rr$ or $\mathrm{dist}(y, \partial\Omega) \leq Rr$. By symmetry, assume the former. 
Then there exists $z \in \partial\Omega$ such that $|x - z| \leq Rr$, and hence $|y - z| \leq (R+1)r$. Using assumption (1), we obtain
\begin{eqnarray}\label{firstcase}
|u(x) - u(y)| &\leq& |u(x) - u(z)| + |u(z) - u(y)|\nonumber \\
&\leq& C_1 (Rr)^\alpha + C_1 ((R+1)r)^\alpha \leq 2C_1 (R+1)^\alpha r^\alpha.
\end{eqnarray}


Now, suppose $x, y \in \Omega_{Rr}$. Let $d = |x - y|$. By assumption (2), we have:
\begin{eqnarray}
  \left|u(x)-u_d(x)\right|&\leq& C_2d^{\alpha},\label{secondcase1}\\
  \left|u(y)-u_d(y)\right|&\leq& C_2d^{\alpha},\label{secondcase2}
\end{eqnarray}
It remains to estimate $|u_d(x) - u_d(y)|$. By definition,
\beq\label{secondcase3}
\left|u_d(x)-u_d(y)\right|=\frac{1}{d^n}\left|\int_{\mathbb{R}^n}(u(x-z)-u(y-z))\eta(\frac{z}{d})\,dz\right|.
\eeq
Define an auxiliary function $g(z) = \frac{1}{2} \displaystyle\inf_{w \in B_1(z)} \eta(w)$ and set 
\[f(z) = \eta(z) - g(z) - g\left(z + \frac{x-y}{d}\right).\] 
By construction, $f \geq 0$. Let $\kappa = \displaystyle\int_{\mathbb{R}^n} g \,d\mu > 0$. Since $\displaystyle\int_{\mathbb{R}^n} \eta \,d\mu = 1$, we have $\displaystyle\int_{\mathbb{R}^n} f \,d\mu = 1 - 2\kappa$.
Under the change of coordinates $z\rightarrow z-x+y$, we have 
\[
\int_{\mathbb{R}^n}(u(x-z)-u(y-z))g\left(\frac{z+x-y}{d}\right)dz=\int_{\mathbb{R}^n}(u(2x-y-z)-u(x-z))g\left(\frac{z}{d}\right)\,dz.
\]
Then we obtain
\begin{align}
&\,\,\,\,\,\,\,\left|\int_{\mathbb{R}^n}(u(x-z)-u(y-z))\eta\left(\frac{z}{d}\right)\,dz\right|\nonumber\\[4pt]
&=\left|\int_{\mathbb{R}^n}(u(x-z)-u(y-z))\left(f\left(\frac{z}{d}\right)+g\left(\frac{z}{d}\right)+g\left(\frac{z+x-y}{d}\right)\right)\,dz\right|\nonumber\\[4pt]
&=\left|\int_{\mathbb{R}^n}(u(x-z)-u(y-z))f\left(\frac{z}{d}\right)+(u(2x-z-y)-u(y-z))g\left(\frac{z}{d}\right)\,dz\right|\nonumber\\[4pt]
&\leq \left|\int_{\mathbb{R}^n}\omega(r)f\left(\frac{z}{d}\right)+\omega(2r)g\left(\frac{z}{d}\right)\,dz\right|.
\end{align}
Substituting back into \eqref{secondcase3} yields
\beq
\left|u_d(x)-u_d(y)\right|\leq \kappa\omega(2r)+(1-2\kappa)\omega(r),
\eeq
Combining this with \eqref{secondcase1} and \eqref{secondcase2}, we obtain
\beq\label{secondcase}
\left|u(x)-u(y)\right|\leq 2C_2r^{\alpha}+\kappa\omega(2r)+(1-2\kappa)\omega(r).
\eeq
Combining the two cases \eqref{firstcase} and \eqref{secondcase}, we derive the key inequality
\beq\label{elemcru}
\omega(r)\leq \max\{2C_1 (R+1)^\alpha r^\alpha, 2C_2r^{\alpha}+\kappa\omega(2r)+(1-2\kappa)\omega(r)\}.
\eeq
We now iterate this inequality. Let
\[C_3=\frac{2C_1 (R+1)^\alpha (diam(\Omega))^{\alpha}}{{\epsilon}_0^{\alpha}}.\]
Note that by (1),
\[\omega({\epsilon}_0)\leq 2C_1 (R+1)^\alpha (diam(\Omega))^{\alpha}=C_3{\epsilon}_0^{\alpha}.\] 
 Define
\beq\label{const-c4}
C_4=\max\left\{C_3,\frac{C_2}{(1-2^{\alpha-1})\kappa}\right\}.
\eeq
We claim that for $r\leq {\epsilon}_0$, 
\beq\label{ite}
\text{if\ } \omega(2r)\leq C_4(2r)^{\alpha},\ \text{then} \ \omega(r)\leq C_4r^{\alpha}.
\eeq 
In fact 
if $\omega(r)\leq 2C_1 (R+1)^\alpha r^{\alpha}$, then the inequality follows immediately; Otherwise 
\beqs
\omega(r)&\leq& 2C_2r^{\alpha}+\kappa\omega(2r)+(1-2\kappa)\omega(r)\\
&\leq& 2C_2r^{\alpha}+\kappa C_4(2r)^{\alpha}+(1-2\kappa)\omega(r),
\eeqs 
i.e.,
\[\omega(r)\leq \frac{2C_2+2^\alpha \kappa C_4}{2\kappa}r^{\alpha}.\]
By the choice of $C_4$ in \eqref{const-c4},
\[\omega(r)\leq C_4 r^{\alpha}.\] 
This proves the claim.

The iteration argument now proceeds standardly. 
For $x,y\in\Omega$, 
if $\left|x-y\right|\geq{\epsilon}_0$
\[\left|u(x)-u(y)\right|\leq C_3\left|x-y\right|^{\alpha}.\] 
otherwise choose an integer $s$ such that 
\[\frac{{\epsilon}_0}{2^s}\leq\left|x-y\right|\leq\frac{{\epsilon}_0}{2^{s-1}}.\] 
By iterating the claim $s$ times starting from $r = \epsilon_0$, we obtain
\[\left|u(x)-u(y)\right|\leq \omega\left(\frac{{\epsilon}_0}{2^{s-1}}\right)\leq C_4 
\left(\frac{{\epsilon}_0}{2^{s-1}}\right)^\alpha\leq 2^\alpha C_4\left|x-y\right|^{\alpha}.\] 
Taking $C=\max\{C_3, 2^\alpha C_4\}$  completes the proof.
\end{proof}

\begin{rem}\label{elementarylemmaremark}
When $\eta=\frac{1}{\omega_n}\chi_{B_1}$, we recover the regularization defined in \eqref{mo2}. 
\end{rem}

\subsection{The manifold case}
A natural generalization of this regularizing function to a Hermitian manifold setting is given by \cite{D}. Let $(X, \omega)$
be a Hermitian manifold and $dim_{\mathbb C}(X)=n$. Let $\eta$ be a smoothing kernel  on $[0,+\infty)$ with compact support in $[0,1]$, such that $\displaystyle\int_{\mathbb{R}^n}\eta(\left|\xi\right|^2)\,d\mu=1$. Let
\beq\label{reg-manifold}
\tilde{u}_{\epsilon}(z)=\frac{1}{\epsilon^{2n}}\int_{\xi\in T_zX}u(\exp_z(\xi)) \eta\left(\frac{|\xi|^2_\omega}{\epsilon^2}\right) \,dV_\omega(\xi),
\eeq
where $\exp_z: T_zX\ni \xi \longrightarrow \exp_z(\xi)\in X$
is the exponential mapping at $z\in X$ and $dV_\omega(\xi)$ is the induced measure $\frac{1}{2^n n!}(dd^c|\xi|_\omega^2)^n$. 
\begin{lem}\label{elementarylemma-manifold}
Let $\Omega \subset X$ be a bounded domain and $u \in L^\infty(\Omega)$. 
  Assume there exist constants $\epsilon_0 > 0$, $C_1, C_2 > 0$, $R>0$, and $\alpha \in (0,1)$ such that the following hold:
\begin{enumerate}
\item [(1)] $\left|u(x)-u(y)\right|\leq C_1\text{dist}(x,y)^{\alpha},\ \ \forall x\in \Omega,\ y\in\pom$;
\item [(2)] $\forall \epsilon\in(0,{\epsilon}_0)$, we have
$$\left|\tilde{u}_{\epsilon}(x)-u(x)\right|\leq C_2{\epsilon}^{\alpha},\ \,\forall x\in {\Omega}_{R\epsilon}.$$
\end{enumerate}
Here $\text{dist}(\cdot,\cdot)$ is the distance with respect to $\omega$ and $\Omega_{\epsilon}=\{x\in\Omega|\text{dist}(x,\pom)>\epsilon\}$.
Then there exists $C>0$, depending only on $n$, $\epsilon_0$, $C_1$, $C_2$, $R$, $\text{diam}(\Omega)$, $\alpha$, $\omega$ and $\eta$, such that  
\[\left|u(x)-u(y)\right|\leq C\text{dist}(x,y)^{\alpha},\ \  \forall x,y\in\Omega.\]
\end{lem}
\begin{proof}
Let $x,y\in\Omega$ such that $\mathrm{dist}(x,y)\leq r$. 
Suppose $\mathrm{dist}(x, \partial\Omega) \leq Rr$ or $\mathrm{dist}(y, \partial\Omega) \leq Rr$.
Then we have 
\[|u(x)-u(y)|\leq 2(R+1)^{\alpha}C_1 r^\alpha\]
as in the flat case.
Suppose $x, y\in\Omega_r$.
Following the argument used in the flat case, one need only estimate the term $|\tilde{u}_{d}(x)-\tilde{u}_{d}(y)|$ for all $x,y\in \Omega_{d}$ with $\operatorname{dist}(x,y)=d\leq r$ in order to bound $\omega(r)$. Let $\gamma$ denote the minimizing geodesic joining $x$ and $y$, and consider the parallel transport map $Tr: T_xX \rightarrow T_yX$ along $\gamma$. Observe that $Tr$ is an isometric transformation. Hence, we may compute
\beqs
&&\left|\tilde{u}_{d}(x)-\tilde{u}_{d}(y)\right|\nonumber\\[4pt]
&=&\frac{1}{d^{2n}}\left|\int_{\xi\in T_xX}u(\exp_x(\xi)) \eta\left(\frac{|\xi|^2_g}{d^2}\right) \,dV_\omega(\xi)-\int_{\xi\in T_yX}u(\exp_y(\xi)) \eta\left(\frac{|\xi|^2_g}{d^2}\right) \,dV_\omega(\xi)\right|\nonumber\\[4pt]
&=&\frac{1}{d^{2n}}\left|\int_{\xi\in T_xX}u(\exp_x(\xi)) \eta\left(\frac{|\xi|^2_g}{d^2}\right) \,dV_\omega(\xi)-\int_{\xi\in T_xX}u(\exp_y(Tr(\xi))) \eta\left(\frac{|\xi|^2_g}{d^2}\right) \,dV_\omega(\xi)\right|\nonumber\\[4pt]
&\leq& \frac{1}{d^{2n}}\left|\int_{\xi\in T_xX}\left(u(\exp_x(\xi))-u(\exp_x(\xi+\log_xy))\right) \eta\left(\frac{|\xi|^2_g}{d^2}\right) \,dV_\omega(\xi)\right|\nonumber\\[4pt]
&&+\frac{1}{d^{2n}}\left|\int_{\xi\in T_xX}\left(u(\exp_x(\xi+\log_xy))-u(\exp_y(Tr(\xi)))\right) \eta\left(\frac{|\xi|^2_g}{d^2}\right) \,dV_\omega(\xi)\right|\nonumber\\[4pt]
&&=:I+II.
\eeqs
By employing the same argument as in the flat case, we deduce the existence of a fixed constant $\kappa \in (0, \frac{1}{2})$ such that
\begin{equation}\label{estimate-I-1}
I\leq \kappa\omega(2s)+(1-2\kappa)\omega(s),
\end{equation}
where 
\[s=\sup_{\xi\in T_xX, \left|\xi\right|\leq 2d}\{\text{dist}(\exp_x(\xi),\exp_x(\xi+\log_xy))\}.\]
We claim that for some constant $C$, 
\begin{equation}\label{estimate-I-2}
s\leq d+Cd^2.
\end{equation}
It suffices to prove this estimate for sufficiently small $d$. To this end, we introduce geodesic normal coordinates centered at $x$ (so that $x$ corresponds to the origin). In these coordinates, radial lines through the origin are geodesics, and the exponential map at the origin is simply the identity: $\exp_x(\xi) = \xi$.
Now consider the map
\begin{eqnarray*}
TX &\longrightarrow&  X\\
(y,\xi)&\longrightarrow & \exp_y(\xi).
\end{eqnarray*} 
Expanding it in a Taylor series around $(0,0)$ and noting that its first-order term must be $y + \xi$ (as it recovers the flat case at leading order), we obtain
\[\exp_y(\xi)=y+\xi+O(\left|y\right|^2+\left|\xi\right|^2).\]
Therefore
\begin{equation}
\exp_x(\xi+\log_xy)=\exp_{\exp_x(\xi)}(v)+O(\left|y\right|^2+\left|\xi\right|^2).
\end{equation}
Here, $\log_x : X \to T_x X$ denotes the inverse of the exponential map, and $v \in T_{\exp_x(\xi)}X$ is the vector whose coordinates in the chosen geodesic normal chart coincide with those of $\log_x y$. When $\operatorname{dist}(x,y) = |\log_x y| = d$ and $|\xi| \le d$, we obtain $|v| = d + O(d^2)$. Consequently,
$s=d+O(d^2)$, which completes the proof of the claim.

Combining estimates \eqref{estimate-I-1} and \eqref{estimate-I-2}, we obtain
\begin{equation}\label{estimate-I}
  I\leq \kappa\omega(2r+Cr^2)+(1-2\kappa)\omega(r+Cr^2).
\end{equation}
To estimate $II$, define two smooth maps $F, G : T_x X \to X$ by
\[F(\xi)=\exp_x(\xi+\log_xy),\ \ G(\xi)=\exp_y(Tr(\xi)).\]
Both $F$ and $G$ are locally invertible, and their difference satisfies 
\beq\label{diff-map}
||F-G||_{C^1}=O(\left|y\right|^2+\left|\xi\right|^2)=O(d^2).
\eeq
Notice that $II$ can be written as
\[II=\int_{X}u(z)K(z)\, dz,\]
where $K(z)$ is the difference of the push-forwards (under $F$ and $G$) of the measure 
\[\frac{1}{d^{2n}}\eta\left(\frac{|\xi|^2_g}{d^2}\right) \,dV_\omega(\xi).\] 
By \eqref{diff-map}, we have $K=O(d^{1-2n})$.
 Moreover, the diameter of $\operatorname{supp}(K)$ is bounded by $O(d)$. Consequently,
 \begin{equation}\label{estimate-II}
II\leq C_3d\leq C_3r.
\end{equation}
for some $C_3>0$.

Combining estimates \eqref{estimate-I} and \eqref{estimate-II}, we obtain
\begin{equation}
\omega(r)\leq  \sup\{2(R+1)^{\alpha}C_1 r^\alpha, \kappa\omega(2r+Cr^2)+(1-2\kappa)\omega(r+Cr^2)+Cr^\alpha+Cr\}.
\end{equation}
For $r \leq \epsilon_0$, this inequality implies
\begin{equation}\label{ite-pre}
\omega(r)\leq \kappa \omega((2+C_4\epsilon_0)r)+(1-2\kappa)\omega((1+C_4\epsilon_0)r)+C_4r^\alpha
\end{equation}
for some $C_4>0$.
We now iterate this inequality. Because the form of \eqref{ite-pre} differs from the flat‑space case, a different iteration argument is required. Without loss of generality, we may assume that
\[\kappa(2+C_4\epsilon_0)^{\alpha}+(1-2\kappa)(1+C_4\epsilon_0)^{\alpha}<1,\]
otherwise we simply replace $\epsilon_0$ by a smaller number. Choose a constant
\[C_5\geq \frac{C_4}{1-\kappa(2+C_4\epsilon_0)^{\alpha}+(1-2\kappa)(1+C_4\epsilon_0)^{\alpha}}.\] 
We then assert the following claim:
\beq\label{ite}
\text{if\ } \omega(s)\leq C_5s^{\alpha} \text{\ for\ } s\geq (1+C\epsilon_0)r,\ \text{then} \ \omega(r)\leq C_5r^{\alpha}.
\eeq 
Indeed, \eqref{ite} follows directly from \eqref{ite-pre} together with the choice of $C_5$.
Since $\omega(\epsilon_0)\leq 2\left|\left|u\right|\right|_{L^\infty}\leq C$, an iteration analogous to the one used in the flat case (now based on \eqref{ite}) yields $\left|\left|u\right|\right|_{C^{\alpha}} \le C$.
\end{proof}

\begin{rem}
One can easily see that \eqref{reg-manifold} is also well-defined on a Riemannian manifold. Although the lemma above is stated in the Hermitian setting, its proof remains valid in the Riemannian setting as well.
\end{rem}

\vskip 20pt

\section{Boundary H\"older continuity}

In this section, we prove estimates near the boundary for solutions to the complex Monge-Amp\`ere equation. We begin by recalling a fundamental $L^{\infty}$-estimate.

\begin{theo}[\cite{K2, WWZ}]\label{WWZThm}
  Let $\Omega\subset\mathbb{C}^n$ be a pseudo-convex domain. Assume $\varphi\in C^0(\overline{\Omega})$, $f\in L^p(\Omega)$, $p>1$, let $u\in PSH(\Omega)$ be solution to the equation
 \begin{equation}
\label{MA}
  \left\{ 
  \begin{alignedat}{2}(dd^cu)^n~& =f\,d\mu~&&\text{\ in} ~\ \ \Omega, \\\
 u ~&=\varphi~&&\text{\ on}~\ \  \pom.
\end{alignedat}
\right.
\end{equation}
  Then for any $0<\delta<\frac{1}{np^*}$(where $\frac{1}{p} + \frac{1}{p^*} = 1$), there is a constant $C>0$ depending on $n$, $p$, $\delta$ and the the diameter of $\Omega$, such that 
   \beq
   |\inf_{\Omega}u|\leq |\inf_\pom\varphi |+C{\left\|f\right\|}_{L^p(\Omega)}^{\frac{1}{n}}\cdot\left|\Omega\right|^\delta.
   \eeq
  where $\left|\Omega\right|$ stands for the volume of $\Omega$.
\end{theo}
Next, we construct an auxiliary function which will be used as a building block for a barrier function near the boundary.
\begin{lem}\label{defofrho}
  Let $\Omega$ be a strictly pseudo-convex smooth domain in $\mathbb{C}^n$ and assume $0\in\pom$. Then there exists a 
  function $\rho\in C^\infty(\mathbb{C}^n)$ and a radius $r_0 > 0$ such that:
  \begin{enumerate}
\item[(1)] $\rho(0) = 0$ and $\rho(z) \geq |z|^2$ for all $z \in \Omega \cap B_{r_0}(0)$;
\item[(2)] $-\rho\in \text{PSH}({\mathbb{C}}^n)$.
\end{enumerate}
\end{lem}
\begin{proof}
  By the strict pseudo-convexity of $\Omega$, 
  there exists a defining function $f \in C^\infty(\mathbb{C}^n)$ which is strictly plurisubharmonic near $0$ such that, for some small $r_0 > 0$,
  \[\Omega\cap B_{r_0}(0)=\{z\in B_{r_0}(0) \ |\ f(z)<0\}.\] 
  The function $f$ has a Taylor series expansion near $0$:
  \[f(z)=\sum\limits_{j=1}^nRe(a_jz_j)+\sum\limits_{i,j=1}^{n}(Re(b_{ij}z_iz_j)+c_{ij}z_iz_{\overline{j}})+O(\left|z\right|^3).\]
  Now, define the function
    \[\rho=-C\left(\sum\limits_{j=1}^nRe(a_jz_j)+\sum\limits_{i,j=1}^{n}Re(b_{ij}z_iz_j)+(c_{ij}-\epsilon\delta_{ij})z_iz_{\overline{j}}\right)\] 
for constants $C > 0$ and $\epsilon > 0$ to be chosen.  
For sufficiently small $\epsilon$, the matrix $(c_{ij}-\epsilon\delta_{ij})$ remains positive definite, ensuring that $-\rho$ is plurisubharmonic.
Furthermore, for $z \in \Omega \cap B_{r_0}(0)$, we have $f(z) < 0$, which implies
\[\sum\limits_{j=1}^nRe(a_jz_j)+\sum\limits_{i,j=1}^{n}Re(b_{ij}z_iz_j)\leq \sum\limits_{i,j=1}^{n}-c_{ij}z_iz_{\overline{j}}+O(\left|z\right|^3).\]
Substituting this into the definition of $\rho$ yields:
  \[\rho(z)\geq C(\epsilon \left|z\right|^2+O(\left|z\right|^3)).\]
By first choosing $\epsilon$ small enough to preserve plurisubharmonicity, and then choosing $C$ sufficiently large and $r_0$ sufficiently small, we can ensure $\rho(z) \geq |z|^2$ for all $z \in \Omega \cap B_{r_0}(0)$.
\end{proof}

We now state and prove the main boundary regularity result.

\begin{lem}\label{BoundaryHolder-manifold}
Let $\Omega$ be a strictly pseudo-convex smooth domain in $(X,\omega)$. Let $f\in L^p(\Omega,\omega^n)$ for some $p>1$
and let $\varphi\in C^{\alpha}(\pom)$ for some $\alpha \in (0,1)$. 
Suppose $u\in W^{2,1}(\Omega)$ is a solution to the Dirichlet problem \eqref{CMA-manifold}. 
Then, for  $\beta=\min\{\beta',\frac{\alpha}{2+\alpha}\}$ with $0<\beta'<\gamma_0$, there 
exists a constant $C$ which depends only on $n$, $p$, $\beta$, $\Omega$, $\left\|f\right\|_{L^p(\Omega,\omega^n)}$ and $\left\|\varphi\right\|_{C^{\alpha}(\pom)}$
 such that  
 \begin{equation}\label{bdy-improve}
 \left|{u(x)-u(y)}\right|\leq C \text{dist}(x,y)^{\beta}, \ \ \forall x\in \Omega, \ y\in \pom.
\end{equation}
\end{lem}

\begin{proof}
We prove the lemma for $\Omega \subset \mathbb{C}^n$; the general case on a manifold follows by working in local coordinate charts covering $\partial\Omega$.

Without loss of generality, assume $y=0\in\pom$ and $u(0)=0$. Let $\rho$ and $r_0$ be  be the function and radius from Lemma \ref{defofrho}. Let $M=\left|\inf_{\Omega}u\right|$. By Theorem \ref{WWZThm}, $M$ is bounded by a constant depending only
on $n$, $p$, $\Omega$, $\left\|f\right\|_{L^p(\Omega)}$, $\left\|\varphi\right\|_{C^0(\pom)}$.
Let 
\[L=\sup_{x,y\in\pom}\frac{\left|\varphi(x)-\varphi(y)\right|}{\left|x-y\right|^{\alpha}}.\]
Fix $x_0\in \Omega$. Our goal is to estimate $|u(x_0)|$.

Let 
\[r=\left|x_0\right|^{\frac{1-\beta}{2}},\ \ 
\epsilon=Lr^{\alpha}, \ \ A=\frac{M}{r^2}.\] 
Here, $\beta \in (0,1)$ is an exponent to be determined later in terms of $\alpha$ and $p$.

If $r\geq r_0$, then $\left|x_0\right|\geq r_0^{\frac{2}{1-\beta}}$, and we have the trivial estimate
\beq\label{bdy-est1}
|u(x_0)|\leq Mr_0^{\frac{-2\beta}{1-\beta}}\left|x_0\right|^{\beta}.
\eeq

Now assume $r\leq r_0$.
Consider the function
\[h(z)=u(z)+\epsilon+A\rho(z)\geq 0.\]
We verify that $h \geq 0$ on $\partial(\Omega \cap B_r(0))$.

\begin{itemize}
\item On $\partial\Omega \cap B_r(0)$: We have $u(z) = \varphi(z)$ and $|\varphi(z)| \leq L |z|^{\alpha} \leq L r^{\alpha} = \epsilon$. Since $\rho(z) \geq 0$ for $z \in \overline{\Omega}$ (by Lemma \ref{defofrho}, as $\rho(z) \geq |z|^2 \geq 0$), it follows that $h(z) \geq \varphi(z) + \epsilon \geq 0$.
\item On $\Omega \cap \partial B_r(0)$: We have $|z|= r$. Since $\rho(z) \geq |z|^2 = r^2$ and $u(z) \geq -M$, we get
\[h(z)\geq -M+Ar^2+\epsilon=-M+M+\epsilon=\epsilon\geq 0.\]
\end{itemize}
Thus, $h \geq 0$ on the boundary of $\Omega \cap B_r(0)$.
Now let $v$ be the solution to
\begin{equation}
\label{LCMA-dir-1}
  \left\{ 
  \begin{alignedat}{2}(dd^cv)^n~& =f\,d\mu~&&\text{\ in} ~\ \ \Omega\cap B_r(0), \\\
 u ~&=0~&&\text{\ on}~\ \  \partial(\Omega\cap B_r(0)).
\end{alignedat}
\right.
\end{equation}
By Theorem \ref{WWZThm}, we obtain the following estimate for $v$ inside $\Omega \cap B_r(0)$.
\begin{equation}\label{abc}
v\geq -C\|f\|_{L^p(\Omega)}^{\frac{1}{n}}\cdot\left|\Omega\cap B_r(0)\right|^{\delta}\geq -C\|f\|_{L^p(\Omega)}^{\frac{1}{n}}\cdot r^{2n\delta}.
\end{equation}
This estimate is valid for any $0 < \delta < \frac{1}{n p^*}$. We now choose $\delta$ to optimize the H\"older exponent. 
Let us set
\[2n\delta=\frac{2\beta}{1-\beta}.\]
The condition $\delta < \frac{1}{n p^*}$ then becomes
\[\frac{2\beta}{1-\beta}< \frac{2}{p^*},\]
i.e.,
\beq\label{beta-1}
\beta<\gamma_0=\frac{p-1}{2p-1}.
\eeq
Thus, under the assumption $\beta < \gamma_0$, we have
\begin{equation}\label{v-estimate}
v(z) \geq - C |f|_{L^p(\Omega)}^{\frac{1}{n}} \cdot r^{\frac{2\beta}{1-\beta}} \quad \text{for all } z \in \Omega \cap B_r(0).
\end{equation}

Since $(dd^c (v-A\rho))^n\geq (dd^c v)^n = (dd^c u)^n = f \,d\mu$ in $\Omega \cap B_r(0)$ and $v \leq 0 \leq h=u+\epsilon+A\rho$ on the boundary, the comparison principle implies that $v-A\rho \leq u+\epsilon$ in $\Omega \cap B_r(0)$. In particular, at the point $x_0$, we have
\begin{eqnarray}
  u(x_0)&\geq& -\epsilon-A\rho(x_0)-C\|f\|_{L^p(\Omega)}^{\frac{1}{n}}\cdot r^{\frac{2\beta}{1-\beta}}\nonumber\\
  &\geq& -Lr^{\alpha}-C\frac{M}{r^2}\left|x_0\right|-Cr^{\frac{2\beta}{1-\beta}}\geq -C\left|x_0\right|^{\beta} \label{bdy-est2}
\end{eqnarray}
provided that $\frac{2\beta}{1-\beta}\leq \alpha$, i.e.,
\beq\label{beta-2}
\beta\leq \frac{\alpha}{2+\alpha}.
\eeq
Here we have used $\rho(x_0)\leq C\left|x_0\right|$.
Combining \eqref{bdy-est1} and \eqref{bdy-est2} we have
\[u(x_0)\leq C\left|x_0\right|^{\beta}.\] 
with $\beta=\min\{\beta',\frac{\alpha}{2+\alpha}\}$ for any $0<\beta'<\gamma_0$ by \eqref{beta-1}, \eqref{beta-2}.
The proof is complete. The constant $C$ depends on the parameters stated in the lemma.
\end{proof}

\begin{rem}\label{remarkonexponent}
In previous works \cite{GKZ, Ch1, Ch2, BKPZ}, the barrier function was constructed as a decomposition into
a vanishing boundary problem 
\[  \left\{ 
  \begin{alignedat}{2}(dd^cv)^n~& =f\,d\mu~&&\text{\ in} ~\ \ B, \\\
 u ~&=0~&&\text{\ on}~\ \  \partial B,
\end{alignedat}
\right.
\]
and a homogeneous problem
\[
  \left\{ 
  \begin{alignedat}{2}(dd^cw)^n~& =0\,d\mu~&&\text{\ in} ~\ \ \Omega, \\\
 w ~&=\varphi-v~&&\text{\ on}~\ \  \pom,
\end{alignedat}
\right.
\]
Here $B$ is a ball containing $\bar\Omega$, and $\tilde f=\begin{cases}
f & \ \ \ \  \text{in $\Omega$,}\\
0 & \ \ \ \  \text{in $B\setminus\Omega$.}
\end{cases}$
Then $v\in C^{2\gamma}(\bar\Omega)$ for $\gamma<\gamma_0$. It follows that $w\in C^{\min\{\frac{\alpha}{2}, \gamma\}}(\bar\Omega)$.
 This approach typically yields a $C^{\min\{\frac{\alpha}{2}, \gamma\}}$-barrier $v+w$ for \eqref{LCMA-dir}.
In our proof above, we give a more direct construction of the barrier which is tightly adapted to the boundary geometry and the boundary values, allowing us to utilize the sharp $L^{\infty}$-estimate of Theorem \ref{WWZThm} more effectively.
\end{rem}

\begin{cor}
The boundary H\"older estimate \eqref{bdy-improve} holds for the exponent
\[\beta=\max\{\min\{\gamma'',\frac{\alpha}{2+\alpha}\},\min\{\frac{\alpha}{2},\gamma'\}\}\] with
$0<\gamma'<\gamma_n$, $0<\gamma''<\gamma_0$. In particular, 
the estimate holds for 
\begin{itemize}
\item $\beta<\gamma_0$ when $\frac{\alpha}{2+\alpha}\geq \gamma_0$; 
\item  $\beta=\frac{\alpha}{2+\alpha}$ when $\gamma_n\leq \frac{\alpha}{2+\alpha}< \gamma_0$;
\item  $\beta=\min\{\frac{\alpha}{2},\gamma'\}$ with
$0<\gamma'<\gamma_n$, when $\frac{\alpha}{2+\alpha}<\gamma_n$.
\end{itemize}
\end{cor}

\vskip 20pt

\section{Proof of Theorem \ref{LCMA-dir}}

In this section, we present the proof of Theorem \ref{LCMA-dir}. A crucial tool is the following stability estimate, first established in $\mathbb{C}^n$ by \cite[Theorem 1.1]{GKZ} and later extended to Hermitian manifolds by \cite{EGZ, GGZ}.

\begin{theo}\cite{GKZ, EGZ, GGZ}\label{prep-holder}
Let $\Omega$ be a relative compact open set in a Hermitian manifold $(X,\omega)$. 
Let $u$, $v$ be bounded plurisubharmonic functions in $\Omega$ 
satisfying $u\geq v$ on $\p \Omega$. 
Assume that 
\[(dd^cu)^n=f\,d\mu,  \  \text{with\ } 0\leq f\in L^p(\Omega) \text{\ for some\ } p>1,\] 
where $\mu$ is the volume form associated to $\omega$. 
Then for $r\geq 1$ and  any $\gamma$ satisfying $0\leq \gamma<\frac{r}{np^*+r}$ (where $1/p + 1/p^* = 1$), we have
\beq\label{vu}
\sup\limits_{\Omega}\{v-u\}\leq C\|\max\{v-u,0\}\|_{L^r(\Omega)}^{\gamma}
\eeq
The constant $C > 0$ depends uniformly on $\gamma$, $\|f\|_{L^p(\Omega)}$ and $\|v\|_{L^{\infty}(\Om)}$. 
\end{theo}

Our proof of Theorem \ref{LCMA-dir} follows the general framework of \cite{GKZ}, 
but we treat three distinct cases separately: the flat case ($\mathbb{C}^n$), the smooth manifold case, and the case of a space with isolated singularities.
 
\subsection{The flat case in $\mathbb{C}^n$} Assume $X$ is $\mathbb C^n$.  We employ the regularization $\hat{u}_\epsilon$ by \eqref{mo2}.
By Lemma \ref{BoundaryHolder-manifold}, for $\beta$ given by \eqref{beta},
we have the boundary estimate
\[\left|\hat{u}_{\epsilon}-u\right|\leq C{\epsilon}^{\beta} \text{\ \ on\ } \partial {\Omega}_{\epsilon},\] 
where the constant $C$ is independent of $\epsilon$. 
Since $\hat{u}_\epsilon$ is plurisubharmonic and majorizes $u$ (by the submean value property), we have $\hat{u}_\epsilon - u \geq 0$ in $\Omega_{\epsilon}$. Note that  $\hat{u}_{\epsilon}\in PSH(\Omega_{\epsilon})$.
Applying the stability estimate (Theorem \ref{prep-holder}) with $r=1$ to functions $\hat{u}_\epsilon$ and $u+C|\epsilon|^\beta$, we obtain
\begin{equation}\label{HR}
\sup_{\Omega_\epsilon} \{\hat u_\epsilon-u-C{\epsilon}^{\beta}\}\leq C\|\max \{\hat u_\epsilon-u-C{\epsilon}^{\beta},0\}\|_{L^1(\Omega_{\epsilon})}^{\gamma}.
\end{equation}
for $0<\gamma<\gamma_n$.
We now estimate the $L^1$-norm on the right-hand side. We compute
\begin{eqnarray}
\|\hat u_{\epsilon}-u\|_{L^1(\Omega_{\epsilon})}
&=& \int_{\Omega_{\epsilon}}(\hat u_{\epsilon}(x)-u(x))\,dx \nonumber\\[4pt]
&=& \int_{\Omega_{\epsilon}}\left(\frac{1}{\omega_{2n}\epsilon ^{2n}}\int_{B_{\epsilon}(x)}(u(y)-u(x))\,dy\right)\,dx\nonumber
\end{eqnarray}
A key observation is that the contribution from the interior cancels out. Precisely, by Fubini's theorem,
\[\int_{\Omega_{\epsilon}}\left(\frac{1}{\omega_{2n}\epsilon ^{2n}}\int_{B_{\epsilon}(x)\cap{\Omega}_{\epsilon}}(u(y)-u(x))\,dy\right)\,dx=0.\]
Therefore, the entire $L^1$ norm comes from the region where $B_{\epsilon}(x)\backslash {\Omega}_{\epsilon} \neq \emptyset$. Thus
\begin{eqnarray}
\left|\left|{\hat{u}}_{\epsilon}-u\right|\right|_{L^1(\Omega_{\epsilon})}&=& \int_{\Omega_{\epsilon}}\left(\frac{1}{\omega_{2n}\epsilon ^{2n}}\int_{B_{\epsilon}(x)\backslash {\Omega}_{\epsilon}}(u(y)-u(x))\,dy\right)\,dx\nonumber\\[4pt]
&\leq& C{\epsilon}^{\beta}\int_{\Omega_{\epsilon}\backslash \Omega_{2\epsilon}}\left(\frac{1}{\omega_{2n}\epsilon ^{2n}}\int_{B_{\epsilon}(x)\backslash {\Omega}_{\epsilon}}\,dy\right)dx\nonumber\\
&\leq& C{\epsilon}^{1+\beta}.
\end{eqnarray}
In the last inequality, the pointwise boundary H\"older estimate(Lemma \ref{BoundaryHolder-manifold}) is used again.
Substituting this estimate into \eqref{HR} yields
\[\sup_{\Omega_\epsilon}\{\hat{u}_{\epsilon}-u\}\leq C{\epsilon}^{\beta}+C{\epsilon}^{(1+\beta)\gamma}\] 
for $\gamma\in (0,\gamma_n)$. By Lemma \ref{elementarylemma},
An application of the elementary Lemma \ref{elementarylemma} then implies that $u \in C^{\alpha'}(\overline{\Omega})$ with exponent $\alpha' = \min\{ \beta, (1+\beta)\gamma \}$.

\subsection{The manifold case} 
Now assume $(X,\omega)$ is a complete Hermitian manifold. 
Let $\Omega$ be a bounded smooth strictly pseudo-convex open subset of $X$, that is, there exists a smooth plurisubharmonic defining function of $\Omega$ in a neighborhood of $\overline{\Omega}$.
We intend to use a similar strategy, but the standard convolution is not available. Instead, we use the regularized function $\tilde{u}_\epsilon$ defined in \eqref{reg-manifold}, which is well-defined on $\Omega_{\epsilon}=\{x\in\Omega|\text{dist}(x,\pom)>\epsilon\}$. Here $\text{dist}(\cdot,\cdot)$ is the distance with respect to $\omega$.
However, the functions $\tilde{u}_{\epsilon}$ are generally not plurisubharmonic in general. 
Following the approach in \cite{BD, DDG+}, we will therefore use the Kiselman transform to construct a plurisubharmonic regularization for the proof of Theorem \ref{LCMA-dir}. The following lemma is adapted from \cite[Lemma 1.12]{BD},  \cite[Lemma 4.1]{KN2} and \cite[Lemma 3.1]{DT}.

\begin{lem}\label{BDLemma}
Let $u \in L^{\infty}(\Omega)$ be a bounded quasi-plurisubharmonic function such that $dd^c u \geq \chi$ for a smooth real $(1,1)$-form $\chi$ on $\Omega$. Let $\tilde{u}_{\epsilon}$ be its regularization defined in \eqref{reg-manifold}. Define the Kiselman-Legendre transform at level $c > 0$ by
\beq\label{Kiselman-trans}
u_{c,\epsilon}=\inf_{t\in (0,\epsilon)} \left\{\tilde{u}_{\epsilon}+Kt^2-K{\epsilon}^2-c\log\left(\frac{t}{\epsilon}\right)\right\},
\eeq
there exists a constant $K > 0$ (depending on the curvature of $\omega$, $\chi$, and $\|u\|_{L^\infty(\Omega)}$) and $\epsilon_0 > 0$ such that for all $\epsilon \in (0, \epsilon_0)$:
\begin{enumerate}
\item  The function $\tilde{u}_{\epsilon}+K{\epsilon}^2$ is increasing in $\epsilon$.
\item  The complex Hessian satisfies the estimate:
\beq\label{psh-uce}
dd^cu_{c,\epsilon}\geq \chi -(A\min\{c,\lambda(z,\epsilon)\}+K{\epsilon})\omega,
\eeq
where $A$ is a lower bound for the bisectional curvature of $\omega$, and
\beq\label{lambdazt}
\lambda(z,t)=\frac{\partial}{\partial \log t}(\tilde{u}_{t}+Kt^2).
\eeq
\end{enumerate}
\end{lem}

We now proceed with the proof of the main theorem in the manifold setting.

\begin{proof}[Proof of Theorem \ref{hold1}]
Following \cite{DDG+}, we write
\[\tilde{u}_{\epsilon}(z)=\int_{x\in X} u(x)\eta\left(\frac{|\log_zx|^2_\omega}{\epsilon^2}\right)\,\frac{dV_\omega(\log_z x)}{\epsilon^{2n}}=\int_{x\in X} u(x) K_\epsilon(z, x),\]
where  $x\rightarrow \xi=\log_z x$ is the inverse of $\xi\to x=\exp_z{\xi}$ and 
\[K_\epsilon(z, x)=\eta\left(\frac{|\log_zx|^2_\omega}{\epsilon^2}\right)\,\frac{dV_\omega(\log_z x)}{\epsilon^{2n}}\]
is the semipositive $(n, n)$-form on $X\times X$ which
is the pull-back of the form $\rho_\epsilon\left(\frac{|\xi|^2_\omega}{\delta^2}\right) \,dV_\omega(\xi)$ by
$(z, x)\to \xi=\log_zx$.

By Lemma \ref{BoundaryHolder-manifold}, we have the boundary estimate
\begin{equation}\label{bdy-ep}
\left|\tilde{u}_{\epsilon}-u\right|\leq C{\epsilon}^{\beta} \text{\ \ on\ } \partial {\Omega}_{\epsilon},
\end{equation} 
where the constant $C$ is independent of $\epsilon$.
Choose $K > 0$ as in Lemma \ref{BDLemma} and fix a constant $c > 0$ (to be determined later). By Lemma \ref{BDLemma}, the Kiselman transform $u_{c,\epsilon}$ satisfies
$$dd^c{u}_{c,\epsilon}\geq 
-(Ac+K{\epsilon})\omega\,\,  \text{in} \,\, \Omega_{\epsilon},$$ 
Let $\rho$ be a strictly plurisubharmonic defining function for $\Omega$ such that $dd^c \rho \geq \omega$. Then the function
$$v(z)=(Ac+K\epsilon)\rho(z)+u_{c,\epsilon}(z)$$
is plurisubharmonic in $\Omega_{\epsilon}$. 
Furthermore,  from the definition \eqref{Kiselman-trans}, we have the bounds
\begin{equation}\label{hat-u}
u-K{\epsilon}^2\leq {u}_{c,\epsilon}\leq \tilde{u}_{\epsilon}\,\,  \text{in} \,\, \Omega_{\epsilon}.
\end{equation}
Since $\rho \leq 0$ in $\Omega$, estimates \eqref{bdy-ep} and \eqref{hat-u} imply that on $\partial \Omega_{\epsilon}$
\[v(z)-u(z)-C\epsilon^{\beta}\leq 0.\]
Applying the stability estimate (Theorem \ref{prep-holder}) with $r=1$ to $v$ and $u + C \epsilon^\beta$ yields
\begin{eqnarray}
&&\sup_{\Omega_\epsilon} \{u_{c,\epsilon}+(Ac+K\epsilon)\rho-u-C\epsilon^{\beta}\}\nonumber\\
&\leq& C'\|\max\{u_{c,\epsilon}+(Ac+K\epsilon)\rho-u-C{\epsilon}^{\beta},0\}\|_{L^1(\Omega_{\epsilon}, \omega^n)}^{\gamma}\nonumber\\
&\leq& C'\|\max\{u_{c,\epsilon}-u-C{\epsilon}^{\beta},0\}\|_{L^1(\Omega_{\epsilon}, \omega^n)}^{\gamma}
\label{HR1}
\end{eqnarray}
for $0<\gamma<\gamma_n$.
We now estimate the $L^1$-norm on the right-hand side. Using \eqref{hat-u}, we have
\begin{eqnarray*}
\left|\left|\max\{u_{c,\epsilon}-u-C\epsilon^{\beta},0\}\right|\right|_{L^1(\Omega_{\epsilon}, \omega^n)}
&\leq& \int_{\Omega_{\epsilon}}(u_{c,\epsilon}-u+K\epsilon^2)\, \omega^n\nonumber\\[4pt]
&\leq& \int_{\Omega_{\epsilon}}(\tilde{u}_{\epsilon}-u+K\epsilon^2)\,\omega^n.
\end{eqnarray*}
Although we cannot cancel all the terms inside $\Omega_{\epsilon}$ like in the $\mathbb{C}^n$ case, the error terms are of $O(\epsilon^2)$. In fact, by the computations in Lemma 2.3 in \cite{DDG+}, we have
\begin{eqnarray*}
\int_{\Omega_{\epsilon}}(\tilde{u}_{\epsilon}-u)\,dV_\omega
&=&\int_{(x,z)\in \Omega\times \Omega_{\epsilon}}(u(x)-u(z))K_{\epsilon}(z,x)\wedge dV_{\omega}(z)\nonumber\\
&=&\int_{(x,z)\in (\Omega\backslash\Omega_{\epsilon})\times \Omega_{\epsilon}}(u(x)-u(z))K_{\epsilon}(z,x)\wedge dV_{\omega}(z)\nonumber\\
&&+\int_{(x,z)\in \Omega_{\epsilon}\times\Omega_{\epsilon}}u(x)(K_{\epsilon}(z,x)\wedge dV_{\omega}(z)-K_{\epsilon}(x,z)\wedge dV_{\omega}(x)).
\end{eqnarray*}
By the boundary H\"older estimate (Lemma \ref{BoundaryHolder-manifold}), 
we have $|u(x) - u(z)| \leq C \epsilon^\beta$ for $x \in \Omega \setminus \Omega_{\epsilon}$ and $z \in \Omega_{\epsilon}$ with $\text{dist}(z, x) = O(\epsilon)$. This implies the first term is bounded by $C\epsilon^{1+\beta}$. 
By \cite[Lemma 2.4]{DDG+}, we have
\begin{equation}
\left|K_{\epsilon}(z,x)\wedge dV_{\omega}(z)-K_{\epsilon}(x,z)\wedge dV_{\omega}(x)\right|\leq C{\epsilon}^{2-2n}dV_{\omega}(z)\wedge dV_{\omega}(x).
\end{equation}
Hence the second term in the expression is bounded by $C{\epsilon}^2$. Therefore
\begin{equation}\label{HR-c}
\left|\left|\max\{u_{c,\epsilon}-u-C\epsilon^{\beta},0\}\right|\right|_{L^1(\Omega_{\epsilon}, \omega^n)}\leq C\epsilon^{1+\beta}.
\end{equation}
Substituting \eqref{HR-c} into \eqref{HR1} gives
\beq\label{estimate1-manifold}
\sup_{\Omega_\epsilon}\{u_{c,\epsilon}+(Ac+K\epsilon)\rho-u\}\leq C{\epsilon}^{\beta}+C{\epsilon}^{(1+\beta)\gamma}
\eeq
for $\gamma\in (0,\gamma_n)$. 

Observe that for any fixed point $z$, as $t\rightarrow 0+$,
\[\tilde{u}_{\epsilon}(z)+Kt^2-K{\epsilon}^2-c\log\left(\frac{t}{\epsilon}\right)\rightarrow +\infty.\] 
Hence there exist a $t_{\text{min}}\in (0,\epsilon]$ such that the infimum is attained, i.e. 
\[u_{c,\epsilon}(z)=\tilde{u}_{t_{\text{min}}}(z)+Kt^2_{\text{min}}-K{\epsilon}^2-c\log \left(\frac{t_{\text{min}}}{\epsilon}\right).\]
From \eqref{estimate1-manifold} and the fact that $\rho(z) \leq 0$, we have:
\[\tilde{u}_{t_{\text{min}}}(z)+Kt^2_{\text{min}}-K{\epsilon}^2-c\log \left(\frac{t_{\text{min}}}{\epsilon}\right)\leq u(z)-(Ac+K\epsilon)\rho(z)+C\epsilon^{\beta}+C\epsilon^{(1+\beta)\gamma}.\] 
By Lemma \ref{BDLemma}, we know $\tilde{u}_{t_{\text{min}}}(z)+Kt^2_{\text{min}}\geq u(z)$. 
Therefore, we obtain
\begin{equation}\label{tmin}
-c\log\left(\frac{t_{\text{min}}}{\epsilon}\right)\leq K\epsilon^2-(Ac+K\epsilon)\rho(z)+C\epsilon^{\beta}+C\epsilon^{(1+\beta)\gamma}.
\end{equation}
Now  we choose $c=\epsilon^{\alpha'}$ where $\alpha'=\min\{\beta,(1+\beta)\gamma\}$. All terms on the right-hand side of \eqref{tmin} are of order $O(\epsilon^{\alpha'})$ or higher. Consequently, there exists a constant $\theta > 0$ such that $t_{\text{min}}\geq \theta \epsilon$ for all sufficiently small $\epsilon$. From the definition of $u_{c,\epsilon}$, this implies
\begin{equation}
u_{c,\epsilon}(z)\geq \tilde{u}_{\theta\epsilon}(z)+K(\theta\epsilon)^2-K\epsilon^2.
\end{equation}
Finally, by \eqref{estimate1-manifold}, we  obtain
\begin{equation}
-K\theta^2\epsilon^2\leq \tilde{u}_{\theta\epsilon}-u\leq C\epsilon^{\alpha'}.
\end{equation}
An application of Lemma \ref{elementarylemma-manifold} concludes the proof, showing $u \in C^{\alpha'}(\overline{\Omega})$.
\end{proof}

\subsection{The case of spaces with isolated singularities}

We now consider the H\"older regularity of the solution when the ambient space is a complex space with isolated singularities. Precisely, let $X$ be a reduced, locally irreducible complex space of dimension $n \geq 1$ with only isolated singularities, denoted $X_{\text{sing}}$. Equip $X$ with a Hermitian metric whose fundamental form is $\beta$, and let $d_\beta$ be the induced distance. Let $\Omega$ be a bounded, strongly pseudoconvex domain in $X$ such that $X_{\text{sing}} \subset \Omega$. Given $\varphi \in C^0(\partial \Omega)$ and $f \in L^p(\Omega)$ with $p > 1$, consider the Dirichlet problem
\begin{equation}
\label{LCMA-dir-singular}
  \left\{ 
  \begin{alignedat}{2}(dd^cu)^n~& =f\,\beta^n~&&\text{\ in} ~\ \ \Omega, \\\
 u ~&=\varphi~&&\text{\ on}~\ \  \pom,
\end{alignedat}
\right.
\end{equation}
The existence, uniqueness, and continuity of the solution were established in \cite{GGZ}. See \cite{CC} for the case of compact K\"ahler spaces. In \cite{G}, it was shown that the solution is Hölder continuous away from the singular points, with an exponent matching that in \cite{Ch1}. We now extend our H\"older estimate to this setting.

\begin{theo}\label{hold1-singular}
The unique solution $u \in \text{PSH}(\Omega) \cap C^0(\overline{\Omega})$ to \eqref{LCMA-dir-singular} is $\alpha'$-H\"older continuous on $\overline{\Omega} \setminus X_{\text{sing}}$, where $\alpha'$ is given by \eqref{beta}.
\end{theo}

\begin{proof}[Proof of Theorem \ref{hold1-singular}]
Fix $\delta > 0$. It suffices to prove that $u$ is $\alpha'$-H\"older continuous on $\overline{\Omega} \setminus B_\delta(X_{\text{sing}})$.
Let $\pi: \widetilde{\Omega} \to \Omega$ be a resolution of singularities. Equip $\widetilde{\Omega}$ with a metric $\theta$ defined by
\[\theta=\pi^*\beta+\eta,\]
where $\eta$ is a smooth non-negative $(1,1)$-form with support in $K = \pi^{-1}(B_\delta(X_{\text{sing}}))$, and which is positive definite in a neighborhood of the exceptional divisor $E$.

The pullback $\pi^* u$ satisfies the following equation on $\widetilde{\Omega}$:
\begin{equation}
  \left\{ 
  \begin{alignedat}{2}(dd^c\pi^*u)^n~& =\pi^*fg\,\theta^n~&&\text{\ in} ~\ \ \tilde\Omega, \\\
 \pi^*u ~&=\pi^*\varphi~&&\text{\ on}~\ \  \partial\tilde\Omega,
\end{alignedat}
\right.
\end{equation}
where $g$ is a bounded non-negative function such that $(\pi^* \beta)^n = g\, \theta^n$.
By Theorem \ref{hold1}, $\pi^* u$ is H\"older continuous with respect to the metric $\theta$ on $\widetilde{\Omega}$. Since $\eta$ is supported in $K$, the distance $d_\theta$ coincides with $d_\beta \circ \pi$ on the set $\pi^{-1}(\Omega \setminus B_\delta(X_{\text{sing}}))$. Therefore, the Hölder continuity of $\pi^* u$ with respect to $d_\theta$ implies the Hölder continuity of $u$ with respect to $d_\beta$ on $\Omega \setminus B_\delta(X_{\text{sing}})$, with the same exponent $\alpha'$.
\end{proof}

\vskip 10pt
{\bf Acknowledgements.} The authors would like to thank Haotong Fu for helpful discussions on Lemma \ref{elementarylemma}.


\begin{thebibliography}{XXX}

\bibitem[BKPZ]{BKPZ} 
Baracco, L., Khanh, T. V., Pinton, S., and Zampieri, G., 
H\"older regularity of the solution to the complex Monge-Amp\`ere equation with $L^p$ density, 
\emph{Calc. Var. Partial Differential Equations} \textbf{55} (2016), no.~4, Art.~74, 8~pp.

\bibitem[BD]{BD}
Berman, R. and Demailly, J.-P.,
Regularity of plurisubharmonic upper envelopes in big cohomology classes,
in \emph{Perspectives in Analysis, Geometry, and Topology}, 
Progr. Math. \textbf{296}, Birkhäuser/Springer, New York, 2012, pp.~39--66.

\bibitem[Br]{Br} 
Bremermann, H. J., 
On a generalized Dirichlet problem for plurisubharmonic functions and pseudo-convex domains: Characterization of Shilov boundaries, 
\emph{Trans. Amer. Math. Soc.} \textbf{91} (1959), 246--276.

\bibitem[BT1]{BT1} 
Bedford, E. and Taylor, B. A.,
The Dirichlet problem for a complex Monge-Amp\`ere equation,
\emph{Invent. Math.} \textbf{37} (1976), no.~1, 1--44.

\bibitem[Ch1]{Ch1}	
Charabati, M.,
H\"older regularity for solutions to complex Monge-Amp\`ere equations,
\emph{Ann. Polon. Math.} \textbf{113} (2015), no.~2, 109--127.

\bibitem[Ch2]{Ch2}	
Charabati, M.,
Regularity of solutions to the Dirichlet problem for Monge-Amp\`ere equations, 
\emph{Indiana Univ. Math. J.} \textbf{66} (2017), no.~6, 2187--2204.

\bibitem[CC]{CC}
Cho, YW. L. and Choi, Y. J., 
Continuity of solutions to complex Monge-Amp\`ere equations on compact K\"ahler spaces, 
\emph{Math. Ann.} (2025). https://doi.org/10.1007/s00208-025-03268-6

\bibitem[D]{D}
Demailly, J.-P., $L^2$ estimates for the $\bar{\partial}$-operator on a semi-positive holomorphic vector bundle over 
a complete K\"ahler manifold, 
\emph{Ann. Sci. \'Ec. Norm. Sup\'er.} (4) 15(3), 457--511 (1982).

\bibitem[DDG+]{DDG+}
Demailly, J.-P., Dinew, S., Guedj, V., Hiep, P. H., Ko\l odziej, S., and Zeriahi, A., 
H\"older continuous solutions to Monge-Amp\`ere equations,
\emph{J. Eur. Math. Soc. (JEMS)} \textbf{16} (2014), no.~4, 619--647.

\bibitem[DT]{DT}
Di Nezza, E. and Trapani, S., 
The regularity of envelopes,
\emph{Ann. Sci. Éc. Norm. Supér. (4)} \textbf{57} (2024), no.~5, 1347--1370.

\bibitem[EGZ]{EGZ}
Eyssidieux, P., Guedj, V., and Zeriahi, A., 
Singular K\"ahler-Einstein metrics, 
\emph{J. Amer. Math. Soc.} \textbf{22} (2009), no.~3, 607--639.

\bibitem[G]{G}
Gonçalves, G. C., 
Modulus of continuity of solutions to complex Monge-Amp\`ere equations on Stein spaces, 
\emph{Indiana Univ. Math. J.}, to appear. 
Preprint, \texttt{arXiv:2405.17242} [math.CV].

\bibitem[GGZ]{GGZ}
Guedj, V., Guenancia, H., and Zeriahi, A., 
Continuity of singular K\"ahler-Einstein potentials, 
\emph{Int. Math. Res. Not. IMRN} (2023), no.~2, 1355--1377.

\bibitem[GKZ]{GKZ}
Guedj, V., Ko\l odziej, S., and Zeriahi, A.,
H\"older continuous solutions to Monge-Amp\`ere equations,
\emph{Bull. Lond. Math. Soc.} \textbf{40} (2008), no.~6, 1070--1080.

\bibitem[K1]{K1}
Ko\l odziej, S.,
Some sufficient conditions for solvability of the Dirichlet problem for the complex Monge-Amp\`ere operator, 
\emph{Ann. Polon. Math.} \textbf{65} (1996), no.~1, 11--21. 
	
\bibitem[K2]{K2}
Ko\l odziej, S., 
The complex Monge-Amp\`ere equation,
\emph{Acta Math.} \textbf{180} (1998), no.~1, 69--117.

\bibitem[KN1]{KN1} 
Ko\l odziej, S. and Nguyen, N.C.,
H\"older continuous solutions of the Monge-Amp\`ere equation on compact Hermitian manifolds,
\emph{Ann. Inst. Fourier (Grenoble)} \textbf{68}(2018), no.~7, 2951--2964. 

\bibitem[KN2]{KN2} 
Ko\l odziej, S. and Nguyen, N.C.,
Stability and regularity of solutions of the Monge-Amp\`ere equation on Hermitian manifolds, 
\emph{Adv. Math.} \textbf{346}(2019),  264--304.

\bibitem[Pl]{Pl}
Pliś, S., 
A counterexample to the regularity of the degenerate complex Monge-Amp\`ere equation, 
\emph{Ann. Polon. Math.} \textbf{86} (2005), no.~2, 171--175.

\bibitem[Wa]{Wa}
Walsh, J. B.,
Continuity of envelopes of plurisubharmonic functions, 
\emph{J. Math. Mech.} \textbf{18} (1968/69), 143--148.

\bibitem[WW]{WW}
Wang, J. X. and Wang, W. L.,
Singular solutions to the complex Monge-Amp\`ere equation, 
\emph{Math. Ann.} \textbf{392} (2025), no.~3, 1503--1518. 

\bibitem[WWZ]{WWZ}
Wang, J. X., Wang, X. J., and Zhou, B., 
A priori estimate for the complex Monge-Amp\`ere equation, 
\emph{Peking Math. J.} \textbf{4} (2021), no.~1, 143--157.
    
\bibitem[Z]{Z}
Zeriahi, A.,
Remarks on the modulus of continuity of subharmonic functions,
Preprint, 2020, \texttt{hal-02901100}.
\end{thebibliography}
\end{document}